\documentclass[a4paper]{amsart}
\usepackage{amssymb, amsmath, amsthm}
\usepackage{stmaryrd, esint}
\usepackage[usenames, dvipsnames]{color}
\newtheorem{theorem}{Theorem}[section]
\newtheorem{lemma}[theorem]{Lemma}

\newtheorem{corollary}[theorem]{Corollary}

\theoremstyle{definition}
\newtheorem{definition}{Definition}[section]

\theoremstyle{remark}
\newtheorem{remark}[definition]{Remark}

\numberwithin{equation}{section}

\newcommand{\abs}[1]{\lvert#1\rvert}

\newcommand{\norm}[1]{\left\lVert#1\right\rVert}

\newcommand{\R}{\mathbb{R}}

\newcommand{\C}{\mathbb{C}}

\newcommand{\N}{\mathbb{N}}

\DeclareMathOperator{\spt}{spt}

\title[non-existence]{Non-existence of a  Wente's $L^\infty$ estimate for the Neumann problem}

\author[J. Hirsch]{Jonas Hirsch}
\address[Jonas Hirsch]{Scuola Internazionale Superiore di Studi Avanzati, via Bonomea, 265, 34136 Trieste, ITALY}
\email{jonas.hirsch@sissa.it}

\begin{document}

\begin{abstract}
We provide a counterexample of Wente's inequality in the context of Neumann boundary conditions. We will also show that Wente's estimates fails for general boundary conditions of Robin type.

\vspace{4pt}

\vspace{4pt}
\end{abstract}

\maketitle
\section{introduction}\label{sec.introduction}

Wente's $L^\infty$-estimate is a fundamental example of a 'gain' of regularity due to the special structure of Jacobian determinants. It concerns the the following Dirichlet problem:
\begin{equation}\label{eq:Dirichlet}\left\{\begin{aligned}
-\Delta u &= f %= \star dv^1\wedge dv^2
 &&\text{ in } D\\
u &= 0 &&\text{ on } \partial D
\end{aligned}\right.
\end{equation}
for the specific choice of  $f=\det( \nabla V)$ with $V \in H^1(D,\R^2)$.
%with $f \in \mathcal{H}^1(D)$ e.g. $f=\det( \nabla V)$ with $V \in H^1(D,\R^2)$.\\
Wente's theorem states:
\begin{theorem}\label{thm.wente}
Let $\Omega \subset \R^2$ be the disc and $f \in \mathcal{H}^1(D)$. Then if $u$ is the unique solution in $W^{1,1}_0(\Omega, \R)$ to \eqref{eq:Dirichlet}, we have the estimate
\[
\norm{u}_{L^\infty(D)} + \norm{du}_{L^2(D)} \le C \norm{\nabla V}_{L^2(D)}^2.
\]
\end{theorem}
Proofs can be found in the original article \cite{Wente}. Later one it had been proven that Wente's inequality hold true under the slightly weaker assumption that $f \in \mathcal{H}^1(D)$, where $\mathcal{H}^1(D)$ is the local Hardy space, see \cite[Definition 1.90]{Semmes}. Proofs can be found for instance in \cite{Helein} and \cite{Topping}. This estimate found many applications, a non complete list includes \cite{Riviere}, \cite{ColdingMinicozzi}, \cite{LammLin}. 

It is natural to ask whether a similar estimate holds true for the Neumann problem:
\begin{equation}\label{eq:Neumann}\left\{\begin{aligned}
-\Delta u &= f% = \star dv^1\wedge dv^2
 &&\text{ in } D\\
\frac{\partial u}{\partial \nu} &= \frac{1}{2\pi} \int_{D} f &&\text{ on } \partial D
\end{aligned}\right.\end{equation}
again for the specific choice of $f=\det( \nabla V)$ with $V \in H^1(D,\R^2)$.\\

The aim of this note is to show that Wente's $L^\infty$ estimate fails for the Neumann problem

\begin{theorem}\label{prop.counterexample}
There exists a sequence $V_n=(a_n,b_n) \in C^\infty(\overline{D}, \R^2)$, $\norm{\nabla V_n}_{L^{2,1}(D)} \le C$ for all $n$ with the property that if $u_n \in W^{1,1}(D)$ are the solutions to \eqref{eq:Neumann} with $f_n=\det(\nabla V_n)$ one has 
\[  \norm{u_n}_{L^\infty(D)}, \norm{ \nabla u_n}_{L^2(D)} \to  + \infty \text{ as } n \to \infty. \]
%There exists $V \in H^1(D, \R^2)$ such that the solution $u \in W^{1,1}$ of \eqref{eq:Neumann} with $f=\det(\nabla V)$ is not in $L^\infty(D)$.
\end{theorem}

More in general we  can extend the above example to more general boundary conditions. Namely we have the following: 

\begin{theorem}\label{prop:Robinboundary}
Let $E \subset \partial D$ be a nonempty union of open intervals, with $0<\mathcal{H}^1(E)<2\pi$ and $\alpha, \beta, \gamma \in \R$ given, with $\alpha >0, \gamma \ge 0$. There exists a sequence $V_n= (a_n, b_n) \in C^\infty(\overline{D},\R^2)$, with $\norm{\nabla V_n}_{L^{2,1}(D)} < C$ with the property that if $u_n \in W^{1,1}(D)$ is the solution to 
\begin{equation}\label{eq:Robin}\left\{\begin{aligned}
-\Delta u_n &= \det(\nabla V_n) %da_n \wedge db_n = \star dv^1\wedge dv^2
 &&\text{ in } D\\
\alpha \frac{\partial u_n}{\partial \nu} +\beta \frac{\partial u_n}{\partial \tau} + \gamma u_n &= 0 &&\text{ on } E \\
u &= 0 &&\text{ on } \partial D \setminus E
\end{aligned}\right.\end{equation}
one has 
\[ \norm{\nabla u_n}_{L^2(D)} \to \infty \text{ as } n \to \infty.\]
\end{theorem}

%That Wente's inequality relies on the Dirichlet data is shown by the extension of the previous proposition to more general boundary data such as mixed Robin boundary conditions. See section \ref{sec.boundary conditions} for the precise statement.\\ 

The paper is organized as follows. In section \ref{sec.knownresults} we collect some known results and a-priori estimates. In section \ref{sec.counterexample} we give the proof of Theorem \ref{prop.counterexample} and in section \ref{sec.boundary conditions} its extension to mixed Robin boundary conditions. \\

While finishing this paper the author became aware that a similar example has been found independently by Francesca Da Lio and Francesco Palmurella, see \cite{DaLioPalmurella}.

\section*{Acknowledgment}
First of all the the author thanks Vicent Millot for proposing to consider M\"obius transformation and this way improving the result. % in the opinion of the author. 
Furthermore he thanks Guido De Philippis for various useful discussions. 
The author is supported by the MIUR SIR-grant Geometric Variational Problems (RBSI14RVEZ).

\section{some known results}\label{sec.knownresults}
Classical solutions to \eqref{eq:Dirichlet}, \eqref{eq:Neumann} have to be understood in the distributional sense.
\begin{definition}\label{def.weak solution}
	A function $u$ is called a solution of the Dirichlet problem if $u \in W^{1,1}_0(D,\R)$ and
\begin{equation}\label{eq:Dirichlet-distributional} \int_D \nabla u \cdot \nabla \psi - f \psi = 0 \text{ for all } \psi \in C^1_0(D).
\end{equation}
A function  $u$ is called a solution of  the Neumann problem if $u \in W^{1,1}(D,\R)$ and
\begin{equation}\label{eq:Neumann-distributional}
\frac{1}{2\pi} \int_D f \,\int_{\partial D} \psi = \int_{D} \nabla u \cdot \nabla \psi - f \psi \text{ for } \psi \in C^\infty_0(\R^2) \text{ for all } \psi \in C^1(\overline{D}).
\end{equation} 
\end{definition}

The Green function for both problems are explicit. For the Dirichlet problem it is
\begin{equation}\label{eq:Green Dirichlet} G_D(x,y)= \frac{1}{2\pi} \ln(\abs{x-y}) - \frac{1}{2\pi} \ln(\abs{y}\abs{x-y^*}) \text{ with } y^*= \frac{y}{\abs{y}^2} \end{equation}
for  Neumann problem it is 
\begin{equation}\label{eq:Green Neumann}	
G_N(x,y)= \frac{1}{2\pi} \ln(\abs{x-y}) + \frac{1}{2\pi} \ln(\abs{y}\abs{x-y^*}) - \frac14 \abs{x}^2 - \frac14 \abs{y}^2.\end{equation}
Using $G_N$ one has the following representation formula
\[ u(y) - \int_D u = - \int_{\partial D} G_N(x,y) \frac{\partial u}{\partial \nu} + \int_D G(x,y) \Delta u \text{ for } u \in C^2(\overline{D}).\]

In terms of existence and uniqueness one has:

\begin{lemma}\label{lem.uniqueness+a priory bound}
	For every $f\in L^1(D)$ there exists a solutions $u_D/ u_N$ to the Dirichlet/ Neumann problem in the sense of Definition \eqref{def.weak solution}. Furthermore the solutions belong to $W^{1,p}(D,\R)$ for every $p <2$, are unique (up to constant in the Neumann problem) and satisfy the estimate
	\begin{equation}\label{eq:apriory}
	\norm{Du}_{L^p(D)} \le C_p \norm{f}_{L^1(D)}.
	\end{equation}
\end{lemma}
\begin{proof}
There are several proofs in the literature treating the case of uniquness and a-priori estimates, compare for instance \cite{LittmanStampacchiaWeinberger}, \cite[Appendix A]{AnconaBrezis}. In our case existence and the a priory estimate \eqref{eq:apriory} can be obtained by using the Green function $G_D, G_N$. Uniqueness for the Dirichlet problem can be obtained by anti symmetric reflection: Let $u$ be a distributional solution of \eqref{eq:Dirichlet-distributional} with $f=0$. One checks that 
\[ \hat{u}(x):=\begin{cases} u(x) & \text{ for } x \in D\\  -u(x^*) & \text{ for } x \notin D \text{ with } x^*= \frac{x}{\abs{x}^2} \end{cases} \]
solves
\[ \int_{\R^2} \nabla \hat{u} \cdot \nabla \psi = \int_D \nabla u . \nabla (\psi(x) - \psi(x^*)) \text{ for all } \psi \in C^1_c(\R^2). \]
But since $\psi(x)-\psi(x^*) \in C^{0,1}_0(D)$ we deduce that $\hat{u}$ is harmonic and therefore smooth in $\R^2$. Now the maximum principle applies since $u$ takes the boundary values in the strong sense.\\
Similarly we deduce in uniqueness in the Neumann problem using the symmetric reflection: Let $v$ be a distributional solution of \eqref{eq:Neumann-distributional} with $f=0$. As before one checks
\[ \hat{v}(x):=\begin{cases} v(x) & \text{ for } x \in D \\ v(x^*) & \text{ for } x \notin D \end{cases} \]
solves
\[ \int_{\R^2} \nabla \hat{v} \cdot \nabla \psi = \int_D \nabla v . \nabla (\psi(x) + \psi(x^*)) \text{ for all } \psi \in C^1_c(\R^2). \]
But since $\psi(x)+ \psi(x^*) \in C^{0,1}(\overline{D})$ we deduce that $\hat{v}$ is harmonic and therefore smooth in $\R^2$. Now the maximum principle implies that $v=const.$ .
\end{proof}

%\section{hardy structures}\label{sec.hardy structures}
%For the sake of completeness and for instance the work of R.Coifman, P.L.Lions, Y.Meyer and S.Semmes \cite{CoifmanLionsMeyerSemmes} we want to mention another important example in the theory of Hardy spaces is the so called the 'div.curl' structure, i.e. $f=E.B$ for $E,B \in L^2(D)$ and $\operatorname{curl} E =0$, $\operatorname{div} B =0$ in the sense of distributions. 
%In two dimension the 'div.curl' structure can be transformed into a Jacobian structure and vice versa using Hodge theory.
%Given $V=(v^1,v^2) \in H^1(D,\R^2)$ set $E:= \nabla v^1$ and $B:=\nabla^\perp v^2 =(\frac{\partial v^2}{\partial x^2}, -\frac{\partial v^2}{\partial x^1}$ we have 
%\[ \det(\nabla V) = \nabla v^1. \nabla^\perp v^2 = E.B. \]
%Vice versa we have by Hodge theory that given $E,B$ as stated then there exists $v^1,v^2 \in H^1(D)$ with $E=\nabla v^1$ and $B=\nabla^\perp v^2$ hence we $V:=(v^1,v^2)$ is the Jacobian equivalent. 
%Both structures come with the estimates
%\begin{align}\label{eq:hardy-estimate for Jacobian and div.curl}
%\norm{ \det{\nabla V}}_{\mathcal{H}^1(D)} &\le C \norm{\nabla V}^2_{L^2(D)}\\
%\norm{ E.B}_{\mathcal{H}^1(D)} &\le C \norm{E}_{L^2(D)}\norm{B}_{L^2(D)}
%\end{align}

\section{Proof of Theorem \ref{prop.counterexample}}\label{sec.counterexample}
In the following we will always identify $\R^2$ with the complex plane $\C$ i.e. $i=e_2$. 
\begin{proof}[Proof of Theorem \ref{prop.counterexample}]
The main step of the proof consists in the following claim:\\
For every $r_0>0$ there exists a sequence $a_n, b_n \in C^\infty(\overline{D})$ with the properties that
\begin{align}\label{eq:blowup-sequence}
&\spt(a_n) \cup \spt(b_n) \subset B_{r_0}(-e_2) \\
	\nonumber& a_n, b_n \rightharpoonup 0 \text{ in } H^1(D)\\\nonumber &	\norm{a_n}_{L^\infty(D)}+\norm{\nabla a_n}_{L^{2,1}(D)}, \norm{b_n}_{L^\infty(D)}+\norm{\nabla b_n}_{L^{2,1}(D)} \le C\\
 &	\norm{ da_n \wedge db_n }_{H^{-1}(D)} \to \infty \text{ as } n \to \infty \label{eq:blowup-sequence-H^-1}. 
\end{align}

Given such a sequence we can conclude the Theorem. Let $u_n$ is the unique solutions to the Dirichlet problem \eqref{eq:Dirichlet} with right hand side $f_n= da_n \wedge db_n$ and $h_n$ is the unique harmonic function satisfying 
\[ \frac{ \partial h_n}{\partial \nu} = \frac{\partial u_n}{\partial \nu} - \frac{1}{2\pi} \int_{\partial D} \frac{\partial u_n}{\partial \nu} \text{ on } \partial D.\]
Such a harmonic function exists since $\int_{\partial D} \left(\frac{\partial u_n}{\partial \nu} - \frac{1}{2\pi} \int_{\partial D} \frac{\partial u_n}{\partial \nu}\right) =0$. % The existence and uniqueness will of $h_n$ will be concluded by the Dirichlet to Neumann map.
It is straight forward to check that \[ v_n:=u_n - h_n \]
is the unique solution to the Neumann problem \eqref{eq:Neumann}. Observe that $v_n$ is a Cauchy sequence in $W^{1,p}(D)$ for all $p<2$ converging to the $v \in W^{1,p}(D)$ the unique solution of \ref{eq:Neumann} with $f= da \wedge db$.  By Wente's theorem we have
\[ \norm{\nabla v_n}_{L^2(D)} \ge \norm{\nabla h_n}_{L^2(D)} - \norm{\nabla u_n}_{L^2(D)} \ge \norm{\nabla h_n}_{L^2(D)} - C\norm{\nabla a_n}_{L^2(D)} \norm{\nabla b_n}_{L^2(D)}.\]
The Theorem follows by showing that \begin{equation}\label{eq:norm of harmonics}\norm{\nabla h_n}_{L^2(D)} \to \infty.\end{equation}

To do so we will use the Dirichlet to Neumann map in the following formulation: Let
\begin{align*}
X_0&:= \{ h \in H^1(D) \colon \Delta h =0 \text{ in } D \text{ and } \fint_D h = 0 \},\\ Y_0&:= \{ u \in H^1(D) \colon \fint_D u = 0 \}.\end{align*}  Endowed with the the $L^2$ inner product $\langle u, v\rangle = \int_D \nabla u \cdot \nabla v$ we obtain Hilbert spaces satisfying $X_0 \subset Y_0$. If we set $Z_0^*:=\{ l \in Y_0^* \colon l(\psi) = 0 \forall \psi \in H^1_0(D)\cap Y_0 \}$ then classical results concerning Dirichlet to Neumann operators imply that the operator 
\[ A: X_0 \to Z_0^* \text{ with } Ah:= \frac{\partial h}{\partial \nu} \]
is a continuous and onto i.e. it has a continuous inverse $A^{-1}$.

Next we identify $\frac{\partial u_n}{\partial \nu} - \frac{1}{2\pi} \int_{\partial D} \frac{\partial u_n}{\partial \nu}$ with an linear functional $l_n \in Y_0^*$ i.e. 
\[  l_n( \psi ):=\int_{\partial D} \left(\frac{\partial u_n}{\partial \nu} - \frac{1}{2\pi} \int_{\partial D} \frac{\partial u_n}{\partial \nu}\right)\, \psi .\]
We will show that they are elements of  $Z_0^*$ with the property that $\norm{l_n}_{H^{-1}(D)} \to + \infty. $
The normal derivative of a solution $u\in W^{1,1}(D)$ to the Dirichlet problem \eqref{eq:Dirichlet}, with $f \in L^1(D)$ is given in the sense of distributions by
\begin{equation}\label{eq:normal derivative}
\int_{\partial D} \frac{\partial u}{\partial \nu} \psi := \int_{D} \nabla u \cdot \nabla \psi - f \psi \text{ for } \psi \in C^1(\overline{D}). 
\end{equation}
The distribution is supported on $\partial D$ since given $\psi_1, \psi_2 \in C^\infty(\overline{D})$ with $\psi_1=\psi_2$ on $\partial D$ we have $\varphi=\psi_1-\psi_2 \in C^1_0(\overline{D})$ with $\varphi = 0$ on $\partial D$ and so by \eqref{eq:Dirichlet-distributional} we have 
\[ \int_{\partial D} \frac{\partial u}{\partial \nu} \varphi =\int_{D} \nabla u \cdot \nabla \varphi - f \varphi =0.\]
By density of $C^\infty_c(D)$ in $H^1_0(D)$ we conclude $l_n(\psi)=0$ for all $\psi \in H^1_0(D)$. 
Furthermore it is straight forward to check that $l_n$ vanishes on the constant functions hence $l_n$ is a well-defined element of $Y_0^*$, since $l_n(\psi) = l_n(\psi - \fint \psi ) $. Hence we conclude that $l_n \in Z_0^*$ for all $n$. The first part of \eqref{eq:normal derivative} and the second part in the definition of $l_n$ are uniformly bounded by Wente's Theorem \ref{thm.wente} because
\begin{align*} \int_{D} \nabla u_n \cdot \nabla \psi &\le \norm{\nabla u_n}_{L^2(D)} \norm{\nabla \psi}_{L^2(D)}\\
\abs{\frac{1}{2\pi} \int_{\partial D}\frac{\partial u_n}{\partial \nu}} & = \abs{ \frac{1}{2\pi} \int_D f_n} \le \frac{1}{2\pi} \norm{\nabla a_n}_{L^2(D)}\norm{\nabla b_n}_{L^2(D)}.\end{align*}  Hence $\norm{l_n}_{H^{-1}(D)} \to \infty$ by \eqref{eq:blowup-sequence-H^-1}. Since $h_n=A^{-1}(l_n)$ and $A^{-1}$ is continuous we conclude \eqref{eq:norm of harmonics}.\\

It remains to construct the sequence $a_n, b_n$ with the properties \eqref{eq:blowup-sequence}, \eqref{eq:blowup-sequence-H^-1}. Performing a translation we can consider the translated disc $D':=D+i$ i.e. $D' \subset H:=\C \cap \{ y \ge 0 \}=\{ r e^{i\theta} \colon 0<\theta < \pi \}$. Furthermore one readily checks that if $\Re(h),\Im(h)$ are real and imaginary part of a holomorphic function $h$ then we have point wise 
\begin{equation}\label{eq:holomorphic function wedge Dirichlet}
d\Re(h)\wedge d\Im(h) = \abs{h'(z)}^2 dx \wedge dy \text{ and } \abs{d\Re(h)}^2= \abs{d\Im(h)}^2= \abs{h'(z)}^2. 
\end{equation}
We will construct our contradicting sequence $a_n, b_n$ as the real and imaginary part of an sequence of holomorphic function $h_n$ on $H$ multiplied by an truncation function $\varphi$.

Indeed consider the family möbius transforms of the complex plane $\C$
\[ m_\epsilon(z):= \frac{z-i \epsilon}{z+ i\epsilon}. \]
We observe that $m_\epsilon$ maps the upper half-space $H$ onto the the disc $D$ for every $\epsilon>0$. 
Furthermore one readily calculates
\begin{equation}\label{eq:derivative and inverse} m'_\epsilon(z) = \frac{2i\epsilon}{(z+i\epsilon)^2} \quad m^{-1}_\epsilon(z)= i\epsilon \frac{z+1}{1-z}. \end{equation}
We note that for every $\delta>0$ one has $m'_\epsilon(z) \to 0$ and $m_\epsilon(z) \to 1$ uniformly on $\C \setminus D_\delta$ for $\epsilon \to 0$. Furthermore $m_\epsilon^{-1} (z) \to 0$ uniformly on $\C \setminus D_\delta(1)$. 
Thus we can conclude that $l_\epsilon:=\abs{m_\epsilon'(z)}^2 dx \wedge dy \to \pi \delta_0$ in the sense of distributions, i.e. given $\psi \in C^0_c(\C)$ arbitrary one has
\begin{align*}
\int_{H} \psi(z) \abs{m_\epsilon'(z)}^2 dx \wedge dy = \int_D \psi \circ m_\epsilon^{-1}(z) \,dx \wedge dy  \to \psi(0) \pi. 
\end{align*}
Furthermore we conclude that if $\varphi$ is any cut of function with $\varphi=1$ in a neighborhood of $0$ we still have $l_\epsilon\lfloor \varphi \to \pi \delta$. Since $\pi \delta_0 \notin H^{-1}(H)$ we conclude that $ \norm{ l_\epsilon\lfloor \varphi}_{H^{-1}(D)} \to \infty$ as $\epsilon \to 0$. 
Fixing a sequence $\epsilon_n \to 0$  then 
\[ a_n:= \varphi \,\Re(m_{\epsilon_n}-1) \text{ and } b_n:= \varphi\, \Im(m_{\epsilon_n}-1) \]
satisfy $ a_n, b_n \in C^\infty(H)$ and $a_n, b_n \to 0$ uniformly in $C^1$ on $\overline{H} \setminus D_\delta$ for any $\delta>0$. Hence for an appropriate choice of $\varphi$ the first two parts of \eqref{eq:blowup-sequence} follows once the uniform bound of the $L^{2,1}$ norm is shown. \\
We calculate \[ da_n \wedge db_n = l_\epsilon \lfloor \varphi^2+ \varphi d\varphi\wedge\left(\Re(m_{\epsilon_n}) d\Im(m_{\epsilon_n}) - \Im(m_{\epsilon_n})d\Re(m_{\epsilon_n})\right) = l_\epsilon \lfloor \varphi^2  + \varphi d \varphi \wedge w_\epsilon. \]
Since $\spt(d \varphi)) \subset \C \setminus D_\delta$ for some $\delta>0$ and $\abs{w_\epsilon} \to 0 $ uniformly on $\C\setminus D_\delta$ we conclude that $\norm{\varphi d \varphi \wedge w_\epsilon}_{H^{-1}} \to 0 $ as $n \to \infty$. Hence $da_n \wedge db_n \to \pi \delta_0$ in the sense of distributions and therefore $\norm{da_n \wedge db_n}_{H^{-1}(H)} \to \infty$ as $n \to \infty$, i.e. \eqref{eq:blowup-sequence-H^-1} holds. 

It remains to show that $\abs{da_n},\abs{db_n}$ are uniformly bounded in $L^{2,1}$. By \eqref{eq:derivative and inverse} we have 
\[ \{ z \in H \colon \abs{m'_\epsilon(z)} \ge t \} = B_{r(t)}(-i\epsilon)\cap H \text{ with } \frac{2\epsilon}{r(t)^2} = t \]
and $\abs{m'_\epsilon}(z) \le \frac{2}{\epsilon} $ for all $z \in H$.
Hence we may estimate \[\mu(t):=\abs{ \{ z \in H \colon \abs{m'_\epsilon(z)} \ge t \}} \le \pi r(t)^2 = \frac{2\epsilon}{t} \pi. \]
Recall that the $L^{2,1}$-norm can be written as 
\[ \norm{f}_{L^{2,1}(H)} = 2 \int_0^\infty  \mu_f(t)^\frac12 \,dt;\]
here $\mu_f(t)= \abs{ \{ z\in H \colon \abs{f(z)} > t \}}$ is the distribution function, compare \cite[Proposition 1.4.9]{Grafakos}.
Using the estimates above we obtain
\begin{align*}
\norm{ \abs{m'_\epsilon}}_{L^{2,1}(H)} \le 2 \sqrt{ 2\pi \epsilon} \int_{0}^{ \frac{2}{\epsilon}} \frac{1}{\sqrt{t}} \, dt \le 8 \sqrt{\pi}.
\end{align*}
Which is uniformly bounded in $\epsilon$ and proofing the last part of \eqref{eq:blowup-sequence}.

\end{proof}

\begin{remark}\label{rem.not in L^infty}
	Observe that if the solution to the Neumann problem is not in $H^1(D)$ then it can neither be in $L^\infty$ nor in $W^{2,1}(D)$. Indeed $u\in W^{2,1}(D)$ would imply $u\in L^\infty$ since $W^{2,1}(D)$ embeds in $L^\infty$ in two dimensions, see for instance Theorem 3.3.10 combined with Theorem 3.3.4 in \cite{Helein}. If $u$ would be in $L^\infty(D)$ then we can take $u_\epsilon \in C^\infty(\overline{D})$ with $u_\epsilon \to u$ in $W^{1,1}(D)$ and uniformly bounded in $L^\infty(D)$. Testing \eqref{eq:Neumann-distributional} with $u_\epsilon$ gives
	\[ \int_D \nabla u \cdot \nabla u_\epsilon = \int_D f u_\epsilon + \frac{1}{2\pi} \int_D f \int_{\partial D} u_\epsilon \le 2 \norm{f}_{L^1} \norm{u_\epsilon}_{L^\infty}. \]
	The right hand side is bounded independent of $\epsilon$ so we conclude that $u \in H^1(D)$ a contradiction.  
\end{remark}

By using more or less an abstract functional analytic arguments we are able to obtain the following Corollary. Its proof is presented in the appendix.

\begin{corollary}\label{cor.single element for neumann}
There exists $a,b \in H^1(D)$ with the additional properties $a,b \in L^\infty(D)$ and $da, db \in L^{2,1}(D)$ such that if $u \in W^{1,1}(D)$ denotes the solution to the Neumann problem \eqref{eq:Neumann} with $f=da \wedge db$ then $u \notin H^1(D)$. 
\end{corollary}

\section{More general boundary conditions}\label{sec.boundary conditions}
Our construction of the counterexample relies mainly on the continuity of the Dirichlet to Neumann map $D_0$. The extension to more general boundary conditions of Robin type follows finding a replacement of the Dirichlet to Neumann map. The replacement is constructed as follows:

%If a similar map exists to imposed boundary conditions one can adapt our construction. Exemplary we want to consider the situation of mixed Robin boundary conditions:
%\begin{proposition}\label{prop:Robinboundary}
%Let $E \subset \partial D$ be a nonempty union of open intervals, with $0<\mathcal{H}^1(E)<2\pi$ and $\alpha, \beta, \gamma \in \R$ given, with $\alpha >0, \gamma \ge 0$. There exists a sequence $V_n= (a_n, b_n) \in C^\infty(\overline{D},\R^2)$, with $\norm{\nabla V_n}_{L^2(D)} < C$ with the property that if $u_n \in W^{1,1}(D)$ is the solution to 
%\begin{equation}\label{eq:Robin}\left\{\begin{aligned}
%-\Delta u_n &= \det(\nabla V_n) %da_n \wedge db_n = \star dv^1\wedge dv^2
% &&\text{ in } D\\
%\alpha \frac{\partial u_n}{\partial \nu} +\beta \frac{\partial u_n}{\partial \tau} + \gamma u_n &= 0 &&\text{ on } E \\
%u &= 0 &&\text{ on } \partial D \setminus E
%\end{aligned}\right.\end{equation}
%one has 
%\[ \norm{\nabla u_n}_{L^2(D)} \to \infty \text{ as } n \to \infty.\]
%\end{proposition}

%The replacement of the Dirichlet to Neumann map is constructed as follows: Let
\begin{align*}
X&:= \{ h \in H^1(D) \colon \Delta h =0 \text{ in } D \text{ and } h = 0 \text{ on } \partial D \setminus E \}\\
Y&:= \{ u \in H^1(D) \colon u =0 \text{ on }  \partial D \setminus E \}
\end{align*}
Since by assumption $\mathcal{H}^1(\partial D \setminus E)>0$ we can endow $X,Y$ with the norm $\norm{u}=\norm{\nabla u}_{L^2(D)}$. Finally we define the closed subset $Z^* \subset Y^*$ by 
\[  Z^*:= \{ l \in Y^* \colon l(u) = 0 \text{ for all } u \in H^1_0(D) \}. \]
Obviously one has the inclusion $X \subset Y$ and $Z^* \subset Y^*$. 
\begin{lemma}\label{lem.boundaryoperator}
The operator $B: X \to Z^*$ defined by
\begin{align*}
 \langle B h, \psi \rangle  &= \int_{\partial D} \left(\alpha \frac{\partial h}{\partial \nu} +\beta \frac{\partial h}{\partial \tau} + \gamma h \right) \psi \\
&:= \alpha \int_{D} \nabla h \cdot \nabla \psi +\beta  \int_{\partial D} \frac{\partial h}{\partial \tau}\psi + \gamma \int_{\partial D } h \psi
\end{align*}
is continuous, onto with continuous inverse $B^{-1}: Z^* \to X$. 
\end{lemma}

\begin{proof}
Instead of $B$ itself we consider the family of operators $B_s: X \to Z^*$ for $s \in [0,1]$. $B_s$ is defined as $B$ with $s\beta, s\gamma$ replacing $\beta, \gamma$. Since $h$ is harmonic in $D$ we have $\langle B_sh, \psi \rangle =0$ for all $\psi \in H^1_0(D)$ by density of $C^\infty_c(D)$ in $H^1_0(D)$. Furthermore we have the estimate
\begin{align*}
\langle B_sh, \psi \rangle &\le \alpha \norm{\nabla h}_{L^2(D)} + \abs{s\beta} \norm{\frac{\partial h}{\partial \tau}}_{H^{-\frac12}{\partial D}} \norm{\psi}_{H^{\frac12}{\partial D}} + s\gamma \norm{h}_{L^2(\partial D)} \norm{\psi}_{L^2(\partial D)}\\
&\le ( \alpha + C \abs{\beta}  + C \gamma) \norm{ \nabla h }_{L^2(D)} \norm{\nabla \psi}_{L^2(D)}. 
\end{align*} 
In the last line we used that for harmonic functions we have \[ \norm{\frac{\partial h}{\partial \tau}}_{H^{-\frac12}(\partial D)} = \norm{\frac{\partial h}{\partial \nu}}_{H^{-\frac12}(\partial D)} = \norm{ \nabla h }_{L^2(D)}\] and the trace theorem for Sobolev functions. \\
This shows that $B_s$ is a family of uniformly bounded operators taking values in $Z^*$. Since $X \subset Y$ we have the lower bound 
\begin{align*}
 \langle B_s h, h \rangle  &= \alpha \int_{D} \nabla h \cdot \nabla h +s\beta \frac12  \int_{\partial D} \frac{\partial h^2}{\partial \tau} + s\gamma \int_{\partial D } h^2\\
 &= \alpha \int_{D} \nabla h \cdot \nabla h + s\gamma \int_{\partial D } h^2 \ge \alpha \norm{ \nabla h }_{L^2(D)}^2. 
\end{align*}
Finally since $B_s = (1-s) B_0 + s B$ the \emph{ method of continuity }, e.g. \cite[Theorem 5.2]{GilbargTrudinger} applies and $B=B_1$ is onto if and only if $B_0$ is onto. By construction we have $B_0h = \alpha \frac{\partial h}{\partial \nu}$ the classical normal derivative on $E$ which is known to be onto by the Dirichlet to Neumann map. This concludes the lemma.

\end{proof}
Now we are able to complete the proof of the Theorem. 
\begin{proof}[Proof of Theorem \ref{prop:Robinboundary}]
	The construction is now essentially the same as in the proof of Theorem \ref{prop.counterexample}. After a rotation we may assume that $-i= - e_2 \in E$. Fix $r_0>0$ such that 
$\partial D \cap B_{r_0}(-i) \subset E$. Let $a_n, b_n, u_n \in C^\infty(\overline{D})$ be the sequence constructed in the proof of Theorem \ref{prop.counterexample}. By the choice of $r_0>0$ we have ensured that
	\[ \spt(a_n) \cup \spt(b_n) \subset B_{r_0}(-i). \]
	Observe that $l_n:= \alpha \frac{\partial u_n}{\partial \nu} + \beta \frac{\partial u_n}{\partial \tau} + \gamma u_n \in Z^*$ because
	\[ \langle Bu_n , \psi \rangle = \alpha \int_{\partial D} \frac{\partial u_n}{\partial\nu} \psi = \alpha \int_D \nabla u_n \cdot \nabla \psi - \alpha \int_D da_n \wedge db_n \psi \]
	and the discussion below \eqref{eq:normal derivative} applies. 
	Furthermore we have
	\[ \norm{l_n}_{Z^*} \ge \alpha \norm{ da_n \wedge db_n }_{H^{-1}(D)} - \alpha \norm{ \nabla u_n}_{L^2(D)}. \]
	By Wente's theorem \ref{thm.wente} $\norm{ \nabla u_n}_{L^2(D)}$ is uniformly bounded and so the applications of lemma \ref{lem.boundaryoperator} gives for $h_n:= B^{-1}(l_n)$ that 
	\[ \norm{ \nabla h_n}_{L^2(D)} \to \infty \text{ as } n \to \infty. \]
	We conclude observing that $v_n:= u_n - h_n$ satisfies the boundary value problem \eqref{eq:Robin} because $u_n = h_n =0 $ on $\partial D \setminus E$ and 
	\begin{equation*}\left\{\begin{aligned}
-\Delta v_n &= - \Delta u_n = da_n \wedge db_n% = \star dv^1\wedge dv^2
 &&\text{ in } D\\
\alpha \frac{\partial v_n}{\partial \nu} +\beta \frac{\partial v_n}{\partial \tau} + \gamma v_n &= l_n - B(h_n) =0 &&\text{ on } E. \end{aligned}\right.\end{equation*}
The blow-up of the $H^{1}$-norm now follows since 
\[ \norm{\nabla v_n}_{L^2(D)}\ge \norm{\nabla h_n}_{L^2(D)}- \norm{\nabla u_n}_{L^2(D)} \to \infty. \]
\end{proof}

As before we obtain as a consequence of Theorem \ref{prop:Robinboundary} the following:

\begin{corollary}\label{cor.single element for robin}
There exists $a,b \in H^1(D)$ with the additional properties $a,b \in L^\infty(D)$ and $da, db \in L^{2,1}(D)$ such that if $u \in W^{1,1}(D)$ denotes the solution to the problem \eqref{eq:Robin} with $f=da \wedge db$ then $u \notin H^1(D)$. 
\end{corollary}
Its combined proof with Corollary \eqref{cor.single element for neumann} can be found in the appendix.

\appendix
\section{abstract functional analytic argument}
Now we want to present the abstract functional analytic argument that leads to Corollary \ref{cor.single element for neumann} and \ref{cor.single element for robin}. We will first proof an "easier" version where every embedding of the involved spaces is linear. Thereafter we show how the same idea translates to our setting. 

\begin{lemma}\label{lem.functional analytic argument}
Given Banach spaces $E_1 \subset E_2$ and $F_1 \subset F_2$ such that the inclusion, $\subset$, corresponds to an continuous embedding. Let $A: E_2 \to F_2$ be a continuous linear operator. Suppose that $F_1$ is a Hilbert space and there is a sequence $\{x_n\}_{n \in \N}$ with the properties that 
\begin{itemize}
\item[(a)] $Ax_n \in F_1$ and $\norm{x_n}_{E_1} \le 1$ for all $n \in \N$;
\item[(b)] $\limsup_{n \to \infty} \norm{Ax_n}_{F_1} = \infty$;
\item[(c)] $f\in F_1 \mapsto \langle Ax_n, f\rangle$ extends to a linear functional $l_n$ on $F_2$ for each $n$. 
\end{itemize}
Then there exists $x \in E_1$ such that $Ax \in F_2 \setminus F_1$ in the sense that there is a sequence $l_n \in F_2^*$ with $\norm{l_n}_{F_1^*} \le 1$ but 
\[ l_n( Ax ) \to \infty. \]
\end{lemma}
\begin{proof}
Passing to a sub sequence we may assume that the $\limsup$ in (b) is actually a limit. 

In a first step we show by induction that there exists $\{y_1, \dotsc , y_n \} \in E_1$ with the properties 
\begin{itemize}
\item[(i)] $\norm{y_i}_{E_1} \le 1$ for all $i$;
\item[(ii)] $\langle A y_i, A y_j \rangle = 0$  if $i \neq j$;
\item[(iii)] $\norm{Ay_i}_{F_1} \ge 2^{2i}$ for all $i$.
\end{itemize}

By $(b)$ there exists $m_1 \in \N$ such that $\norm{Ax_{m_1}}\ge 4$. Hence we may set $y_1:= x_{m_1}$. 

Now suppose $\{y_1, \dotsc y_n\}$ have been chosen. We define the linear continuous operator $P_n: F_1 \to F_1$ by
\[ P_n:= \sum_{i=1}^n \frac{ Ay_i \otimes Ay_i}{\norm{Ay_i}^2}. \] 
It is obvious that $P_n= P_n^t$ and $(ii)$ implies that $P_n^2=P_n$ i.e. $P_n$ is the orthogonal projection onto the finite dimensional space $V_n:=\operatorname{span}\{ Ay_1, \dotsc, Ay_n \}$. Hence $(P_n A): E_1 \to V_n$ is a continuous linear operator onto a finite dimensional vector space. Let $(P_n A)^{-1}: V_n \to \operatorname{span}\{ y_1, \dotsc y_n \} $ denote the inverse of the operator $(P_n A)$ restricted to the finite dimensional space $\operatorname{span}\{ y_1, \dotsc y_n \} $. We may define now the operator
\[ Q_n: E_1 \to E_1 \quad Q_n:= (P_n A)^{-1} \circ (P_n A). \]
We note that $Q_n$ is continuous and $Q_n^2 = Q_n$ hence $Q_n$ is a projection operator.  As a direct consequence we have as well that $(I - Q_n)$ is a continuous projection operator, here $I$ denotes the identity map on $E_2$.%, i.e. $\norm{ I - Q_n}_{E_1} \le C_n$ for some $1\le C_n < \infty$.

By construction we have 
\begin{equation}\label{eq:the combined property of Q_n, P_n} P_n \, A \,( I - Q_n) =0. \end{equation}

The range of $Q_n$ is finite $(A Q_n)$ is a continuous operator and therefore \[ \limsup_{m \to \infty} \norm{ (A Q_n) x_m }_{F_1} < \infty.\] Hence we have
\[ \lim_{m \to \infty} \norm{ A(I-Q_n) x_m }_{F_1} \ge \lim_{m \to \infty} \norm{A x_m}_{F_1}  - \limsup_{m \to \infty} \norm{ (A Q_n) x_m}_{F_1} =\infty. \]
Thus there exists ${m_{n+1}} \in \N$ such that 
\[ \norm{ A(I-Q_n) x_{m_{n+1}} }_{F_1} > 2^{2(n+1)} \norm{ I - Q_n}. \]
We define $y_{n+1} = \frac{ (I-Q_n) x_{m_{n+1}}}{\norm{I-Q_n}}$. Clearly we have $\norm{y_{n+1}}_{E_1} \le 1$ and (iii) holds by the choice of $m_{n+1}$. Finally (ii) follows using that $P_n$ is a orthogonal projection, $Q_n$ a projection and \eqref{eq:the combined property of Q_n, P_n}:
\[ \langle A y_i, Ay_{n+1} \rangle = \langle P_n A y_i, A (I-Q_n) y_{n+1} \rangle = \langle P_n A y_i, ( P_n A (I-Q_n)) y_{n+1} \rangle = 0. \]

Having the sequence $\{y_i\}_{i \in \N}$ to our disposal we obtain $x$ as follows: For each $n$ we define the elements $z_n \in E_1$ and $f_n \in F_1$ by
\[ z_n:= \sum_{i=1}^n 2^{-i} y_i \text{ and } f_n:=\sum_{i=1}^n 2^{-i} \frac{ Ay_i}{\norm{Ay_i}_{F_1}}. \]
Since $E_1, F_1$ are Banachspaces we have that their limes exists: $z=\lim_{n\to \infty} z_n = \sum_{i=1}^\infty 2^{-i} y_i \in E_1$ and $f=\lim_{n\to \infty} f_n = \sum_{i=1}^\infty 2^{-i} \frac{ Ay_i}{\norm{Ay_i}_{F_1}}$. 

Assumption (c) implies that for each $i \in \N$ the map $f \in F_1 \mapsto \langle \frac{A y_i}{\norm{Ay_i}_{F_1}}, f\rangle$ extends to a continuous linear functional $l_i \in F_1^*$. Therefore the continuous linear functional $L_n:= \sum_{i=1}^n 2^{-i} l_i$ has the desired properties using (i)-(iii) since
\begin{align*} L_n(Az)&= \lim_{m \to \infty} L_n(Az_m) = \lim_{m \to \infty} \langle f_n , Az_m \rangle\\
&= \lim_{m \to \infty} \sum_{i=1}^n \sum_{j=1}^m 2^{-i-j} \langle \frac{A y_i}{\norm{Ay_i}_{F_1}}, Ay_j \rangle = \sum_{i=1}^n 2^{-2i} \norm{Ay_i}_{F_1} \ge n .\end{align*}   
This completes the proof.
\end{proof}

Observe that we could directly apply the above result with the following choice of spaces: let $E_1= \mathcal{H}_{loc}^1(D)$ be the local Hardy space of the disk; $E_2= L^1(D)$, $F_1 =\{ f \in H^1(D) \colon \fint_D f = 0 \}$ and $F_2 = W^{1,1}(D)$. But this would not give single elements $a, b \in H^1(D)$ as stated in the Corollaries \ref{cor.single element for neumann}, \ref{cor.single element for robin}.

%Let us now give the proof for this result.
\begin{proof}[Proof of Corollary \ref{cor.single element for neumann} and \ref{cor.single element for robin}]
We introduce the space \[X:=\{ h \in H^1(D) \colon \fint_D h =0 \text{ and } dh \in L^{2,1}(D)\}.\] It becomes a complete Banach space with respect to the norm $\norm{h}_X:= \norm{dh}_{L^{2,1}}$. Furthermore as suggested before we set $E_2:=L^1(D), F_1:=H^1(D), F_2 = W^{1,1}(D)$. Observe that we have a 'bi-linear' linear embedding of $X \times X \hookrightarrow E_2$ by $(h,k) \mapsto dh \wedge dk$ with $\norm{dh \wedge dk}_{L^1}  \le \norm{dh}_{L^{2,1}}\norm{dk}_{L^{2,1}}$. %$\norm{dh \wedge dk}_{L^1(D)} \le \norm{dh}_{L^2}\norm{dk}_{L^2} \le \norm{dh}_{L^{2,1}}\norm{dk}_{L^{2,1}}$

The construction of $(a,b)$ out of the contradicting sequence is the same in case of Neumann or Robin type boundary condition. Hence we will give a simultaneous proof for both. We denote by $A: L^1(D) \to W^{1,1}(D)$ the solution operator to problem \eqref{eq:Neumann} or problem \eqref{eq:Robin} respectively. Recall that by classical elliptic theory there is a constant $C_A >0$ such that $\norm{Ax}_{H^1} \le C_A \norm{x}_{L^2}$.

Let $(a_n,b_n)\in C^\infty(\overline{D}, \R^2)$ be the corresponding contradicting sequence of Theorem \ref{prop.counterexample} or Theorem \ref{prop:Robinboundary}. Without loss of generality we may assume that $\fint a_n = 0 = \fint b_n$ for all $n$, hence $a_n, b_n \in X$. From now on we do not have to distinguish the cases anymore.

We will now proceed very similar as in Lemma \ref{lem.functional analytic argument}. 
By induction we show the existence of a sequence $\{ y_1, y_2, \dotsc, y_n\} \in L^1(D) \cap C^\infty(\overline{D})$ with the properties 

\begin{itemize}
\item[(i)] $\norm{y_i}_{L^1} \le 1$ for all $i$;
\item[(ii)] $\langle A y_i, A y_j \rangle = 0$  if $i \neq j$;
\item[(iii)] $\norm{Ay_i}_{F_1} \ge 2^{3i}$ for all $i$.
\end{itemize}
Simultaneously we will construct a sequence of tuples $(h_i,k_i) \in X\cap C^\infty(\overline{D}) \times  X\cap C^\infty(\overline{D}), i=1, \dotsc, n$ s.t. 
\begin{itemize}
\item[(1)] $\norm{h_i}_{L^\infty} + \norm{dh_i}_{L^{2,1}}+\norm{k_i}_{L^\infty}+\norm{dk_i}_{L^{2,1}} \le 1$
\item[(2)] $dh_i \wedge dk_i = y_i + R_i$ with $\norm{R_i}_{L^2} \le 1$
\item[(3)] $\norm{dh_i}_{L^2} + \norm{dk_i}_{L^2} \le \left(1 + \sum_{j<i} \norm{dh_j}_{L^\infty} + \norm{dk_j}_{L^\infty}\right)^{-1}$. 
\end{itemize}

We start the induction by choosing $(a_1,b_1)$ in the contradicting sequence such that $\norm{ A (da_1 \wedge db_1)}>2^2$. We set $y_1= da_1 \wedge db_1$ and $(h_1, k_1) = (a_1, b_1)$. All properties are clearly satisfied ($R_1  =0$).

Now suppose that we have chosen $y_i, (h_i,k_i)$ for $i=1, \dotsc, n$. We want to construct $y_{n+1}$ and the tuple $(h_{n+1}, k_{n+1})$. 
As in the previous lemma we define the projection operators 
\[ P_n:= \sum_{i=1}^n \frac{ Ay_i \otimes Ay_i}{\norm{Ay_i}^2}; \quad Q_n:= (P_nA)^{-1} (P_nA). \] Here $(P_nA)^{-1}$ denotes as before the inverse of $(P_nA)$ if restricted to the space $\operatorname{span}\{ y_1, \dotsc, y_n\}$. Hence 
for all $x \in L^1(D)$ we have $Q_nx = \sum_{i=1}^n \alpha_i y_i$ and the existence of a constant $C_n>0$ such that $\sum_{i=1}^n \abs{\alpha_i} \le C_n$ for all $x \in L^1(D)$ with $\norm{x}_{L^1} \le 1$. Furthermore due the properties of the contradicting sequence there exist $m \in \N$ such that 
\[ \norm{ A(I-Q_n) da_m \wedge db_m}_{H^1} \ge 2^{3(n+1)} C^2_n\left(n+3 + \sum_{j\le n} \norm{dh_j}_{L^\infty} + \norm{dk_j}_{L^\infty}\right)^2. \]
Let $Q_n da_m \wedge db_m = \sum_{i=1}^n \alpha_i y_i$, and define the elements
\[ \tilde{y}_{n+1}:= (I-Q_n) da_m \wedge db_m \quad \tilde{h}_{n+1}:= a_m - \sum_{i=1}^n \alpha_i h_i \quad \tilde{k}_{n+1} := b_m + \sum_{i=1}^n k_i. \]
We calculate 
\begin{align*}
d\tilde{h}_{n+1}\wedge d\tilde{k}_{n+1} =& da_m \wedge db_m  - \sum_{i=1}^n \alpha_i dh_i \wedge dk_i\\& + d\left(- \sum_{i=1}^n \alpha_i h_i \right) \wedge db_m + da_m \wedge \left(\sum_{i=1}^n k_i \right)\\& - \left( \sum_{i<j} (\alpha_i dh_i \wedge dk_j + \alpha_j dh_j \wedge dk_i \right)\\ 
\overset{(2)}{=}& da_m \wedge db_m  - \sum_{i=1}^n \alpha_i y_i - \sum_{i=1}^n \alpha_i R_i + (I) +(II) +(III).
\end{align*}
We estimate the size of the reminder terms in $L^2(D)$: Due to (2), we have $\norm{\sum_{i=1}^n \alpha_i R_i}_{L^2} \le C_n$. The terms $(I), (II)$ are similarly estimated by 
\begin{align*} \norm{d\left(- \sum_{i=1}^n \alpha_i h_i \right) \wedge db_m}_{L^2} &\le \left(\sum_{i=1}^n \abs{\alpha_i} \norm{dh_i}_{L^\infty} \right) \norm{db_m}_{L^2} \\
\norm{da_m \wedge d\left(\sum_{i=1}^n k_i \right) }_{L^2} &\le \left(\sum_{i=1}^n \norm{dk_i}_{L^\infty} \right) \norm{da_m}_{L^2}.
\end{align*}
Adding both we obtain $\norm{(I)}_{L^2}+\norm{(II)}_{L^2} \le C_n\left(1 + \sum_{j\le n} \norm{dh_j}_{L^\infty} + \norm{dk_j}_{L^\infty}\right)$.
The last term can be estimated as well using only property (3) by 
\begin{align*} \norm{(III)}_{L^2} &\le \sum_{i=1}^n \abs{\alpha_i} \norm{dh_i}_{L^2}\left( \sum_{j<i} \norm{dk_j}_{L^\infty} \right) +   \norm{dk_i}_{L^2}\left( \sum_{j<i} \abs{\alpha_j}\norm{dh_j}_{L^\infty} \right)\\
&\le \left(\sum_{i=1}^n \abs{\alpha_i}\right) + \sup_{j\le n} \abs{\alpha_j} n \le (n+1) C_n.
\end{align*}
Hence we found that $\norm{\tilde{R}_{n+1}}_{L^2} \le C_n \left(n+3 + \sum_{j\le n} \norm{dh_j}_{L^\infty} + \norm{dk_j}_{L^\infty}\right)$ where 
$\tilde{R}_{n+1} = - \sum_{i=1}^n \alpha_i R_i + (I)+ (II) + (III)$ and
\[ d\tilde{h}_{n+1} \wedge d\tilde{k}_{n+1} = (I-Q_n)da_m \wedge db_m + \tilde{R}_{n+1} = \tilde{y}_{n+1} + \tilde{R}_{n+1}.\] 
The desired functions are now simply $y_{n+1}= \frac{\tilde{y}_{n+1}}{\lambda_n}, h_{n+1}= \frac{\tilde{h}_{n+1}}{\lambda_n}, k_{n+1}= \frac{\tilde{k}_{n+1}}{\lambda_n}$ with $\lambda_n = C_n \left(n+3 + \sum_{j\le n} \norm{dh_j}_{L^\infty} + \norm{dk_j}_{L^\infty}\right)$.  
%\[ y_{n+1} = \frac{\tilde{y}_{n+1}}{ C_n^2 \left(n+3 + \sum_{j\le n} \norm{dh_j}_{L^\infty} + \norm{dk_j}_{L^\infty}\right)^2} \quad h_{n+1} = \frac{\tilde{h}_{n+1}}{ C_n \left(n+3 + \sum_{j\le n} \norm{dh_j}_{L^\infty} + \norm{dk_j}_{L^\infty}\right)} \quad k_{n+1} = \frac{\tilde{k}_{n+1}}{C_n \left(n+3 + \sum_{j\le n} \norm{dh_j}_{L^\infty} + \norm{dk_j}_{L^\infty}\right)}\]

Having established the existence of the sequences $y_i, h_i, k_i$ with the claimed properties we construct $a,b\in X$ and a sequence $f_n \in H^1(D)=F_1$ very similar as in the proof to Lemma \ref{lem.functional analytic argument}:
Due to (1) and the fact that $X$ is a complete Banach space we can define elements
\[ a:= \sum_{i=1}^\infty 2^{-i} h_i, \quad b:= \sum_{i=1}^\infty 2^{-i} k_i. \]
Furthermore for each $n \in \N$ let 
\[ f_n:= \sum_{i=1}^n 2^{-i} \frac{ A{y_i}}{\norm{ Ay_i}_{H^1}}. \]
Observe that $f_n$ is a finite sum of $C^1$-functions, hence itself $C^1$ and can therefore be considered as an element of $(L^1)^*=L^\infty$. 
It remains to check that $\lim_{n \to \infty} \int_{D} f_n A(da \wedge db) = + \infty$. 
We have
\[ A(da \wedge db) = \lim_{m \to \infty} \sum_{i=1}^m 2^{-2i} A(dh_i \wedge dk_i)  + \sum_{i<j}^m 2^{-i-j} A\left( dh_i \wedge dk_j + dh_j \wedge dk_i\right). \]
Using (2) we estimate 
\begin{align*} \langle \frac{Ay_k}{\norm{Ay_k}_{H^1}} , A(dh_i \wedge dk_i) \rangle &= \langle \frac{Ay_k}{\norm{Ay_k}_{H^1}}, Ay_i + AR_i \rangle\\ &\ge \delta_{ki} \norm{Ay_i}_{H^1} - C_A \norm{R_i}_{L^2} \ge \delta_{ki} \norm{Ay_i}_{H^1} - C_A. \end{align*}
Hence 
\[  \sum_{i=1}^m 2^{-2i} \langle \frac{Ay_k}{\norm{Ay_k}_{H^1}} , A(dh_i \wedge dk_i) \rangle \ge 2^{-2k} \norm{Ay_k}_{H^1} - \lim_{m\to\infty} \sum_{i=1}^m 2^{-2i} C_A \ge 2^{k} - C_A.\] 
Using (3) we get 
%for $i>j$
%\[ \norm{ A\left( dh_i \wedge dk_j + dh_j \wedge dk_i\right)}_{H^1} \le C_A \left( \norm{dh_i}_{L^2} \norm{dk_j}_{L^\infty} + \norm{dh_j}_{L^\infty} \norm{dk_i}_{L^2} \right)\]
%and therefore 
\begin{align*}
&\sum_{i<j}^m 2^{-i-j} \norm{ A\left( dh_i \wedge dk_j + dh_j \wedge dk_i\right)}_{H^1} \\
&\le C_A \sum_{i<j}^m 2^{-i-j} \left( \norm{dh_i}_{L^2} \norm{dk_j}_{L^\infty} + \norm{dh_j}_{L^\infty} \norm{dk_i}_{L^2} \right)\\
&\le C_A \sum_{i=1}^m 2^{-i} 2 \le 2 C_A.
\end{align*}
Finally combining both we obtain
\[ \langle \frac{Ay_k}{\norm{Ay_k}_{H^1}} , A(da \wedge db) \rangle \ge 2^{k} - 3C_A. \] 
This completes the estimate since 
\begin{align*}
\int_D  f_n A(da\wedge db) = \sum_{k=1}^n 2^{-k} \langle \frac{Ay_k}{\norm{Ay_k}_{H^1}}, A(da \wedge db) \rangle \ge n - 3 C_A.
\end{align*}
\end{proof}

\bibliographystyle{plain}
\bibliography{Wente_counterexample.bib}

\end{document}